\documentclass[11pt,reqno]{amsart}

\usepackage[a4paper,margin=2.5cm]{geometry}
\usepackage{amsmath,amssymb,amsthm,mathtools}
\usepackage{xcolor}
\usepackage{hyperref}
\usepackage{url}

\hypersetup{
  colorlinks=true,
  linkcolor=blue!50!black,
  citecolor=blue!50!black,
  urlcolor=blue!50!black
}

\newcommand{\R}{\mathbb{R}}
\newcommand{\T}{\mathbb{T}}
\newcommand{\E}{\mathbb{E}}
\newcommand{\Pp}{\mathbb{P}}
\newcommand{\dd}{\,\mathrm{d}}
\newcommand{\norm}[1]{\left\lVert #1\right\rVert}
\newcommand{\abs}[1]{\left\lvert #1\right\rvert}
\newcommand{\Leray}{\mathbb{P}}
\DeclareMathOperator{\diver}{div}

\DeclareMathOperator{\rank}{rank}
\DeclareMathOperator{\tr}{tr}

\def\Xint#1{\mathchoice
  {\XXint\displaystyle\textstyle{#1}}%
  {\XXint\textstyle\scriptstyle{#1}}%
  {\XXint\scriptstyle\scriptscriptstyle{#1}}%
  {\XXint\scriptscriptstyle\scriptscriptstyle{#1}}%
  \!\int}
\def\XXint#1#2#3{{\setbox0=\hbox{$#1{#2#3}{\int}$}
  \vcenter{\hbox{$#2#3$}}\kern-.5\wd0}}
\def\fint{\Xint-}

\newtheorem{theorem}{Theorem}[section]
\newtheorem{proposition}[theorem]{Proposition}
\newtheorem{lemma}[theorem]{Lemma}

\theoremstyle{definition}

\newtheorem{remark}[theorem]{Remark}

\begin{document}

\title[Mean--covariance dynamics of the stochastic Weber field]{Exact mean--covariance dynamics of the Weber field in the stochastic Lagrangian representation of the 3D Navier--Stokes equations}

\author{Triphop Mahithitarmmatorn}
\email{triphop.mahi@gmail.com}

\date{\today}


\begin{abstract}
The Constantin--Iyer formula represents a smooth solution of the incompressible
Navier--Stokes equations on $\T^3$ as
$u=\Leray\,\E[(\nabla A_t)^\top(u_0\circ A_t)]$, the projected expectation of a stochastic
Weber field. We separate this expectation into the product of the means and a centred
covariance, and show that the covariance --- together with the second moment of the inverse
deformation gradient and the two-point covariance --- satisfies a \emph{closed}
deterministic advection--diffusion equation: because every Lagrangian label is driven by
the same Brownian path, the quadratic covariation collapses into a Laplacian and no moment
hierarchy appears. From these equations we obtain an energy/production balance, pointwise
and H\"older covariance bounds, and an exact degeneracy of the two-point operator in the
separation variable. A second group of results concerns the mean displacement, which
solves a forced heat equation along the flow; its $L^2$ norm is controlled unconditionally
by the initial energy, while an exact family of Navier--Stokes shear flows shows that no
analogous H\"older bound can hold uniformly down to $t=0$. We further isolate several
obstructions, each backed by an explicit counterexample --- the norm of an expected
deformation gradient does not control the expected norm, even for genuine stochastic
flows of smooth divergence-free drifts; a spatial
H\"older bound on the mean does not imply parabolic Campanato decay --- and we prove a
gauge-invariant quotient criterion equivalent to the Serrin norm of the velocity, together
with a Liouville-type rigidity theorem for ancient solutions whose mean coordinate is
affine of rank at least one. No claim is made regarding global regularity; the missing
steps are stated explicitly as open problems.
\end{abstract}

\maketitle

\section{Introduction}

Let $\T^3=\R^3/\mathbb Z^3$ and $\nu>0$, and let $u$ be a smooth divergence-free solution
of the incompressible Navier--Stokes equations
\begin{equation}\label{eq:NS}
  \partial_t u+(u\cdot\nabla)u-\nu\Delta u+\nabla p=0,\qquad \diver u=0,
\end{equation}
on a time interval $[0,T]$, with initial datum $u_0$. Constantin and Iyer
\cite{ConstantinIyer2008} associate to $u$ the stochastic flow
\begin{equation}\label{eq:flow}
  \dd X_t(a)=u(t,X_t(a))\dd t+\sqrt{2\nu}\,\dd W_t,\qquad A_t=X_t^{-1},
\end{equation}
driven by a single Brownian motion $W$ common to all labels $a$, and prove the
representation
\begin{equation}\label{eq:CI}
  u(t)=\Leray\,\E\!\left[(\nabla A_t)^{\top}(u_0\circ A_t)\right],
\end{equation}
where $\Leray$ denotes the Leray projection. The random field
$F:=(\nabla A_t)^\top(u_0\circ A_t)$ is the stochastic Weber field; its inviscid
counterpart is the classical Weber formula for the Euler equations, and the stochastic
perturbation of Lagrangian trajectories underlying \eqref{eq:flow} was introduced in
\cite{Iyer2006CMP}.

Formula \eqref{eq:CI} expresses the velocity as the expectation of a random field, and it
is tempting to hope that the \emph{mean} deformation $\E\nabla A_t$ alone might control the
solution. It cannot, as we show below, and the reason is instructive: the expectation in
\eqref{eq:CI} couples the stretching directions of $\nabla A_t$ to the values of the
initial datum sampled along the same random trajectory. The natural object separating
these effects is the centred covariance of the Weber field, and the starting point of this
paper is the observation that this covariance obeys a closed, exactly computable
deterministic equation. Throughout we write
\[
  J_t:=\nabla A_t,\quad V_t:=u_0\circ A_t,\quad
  M:=\E J_t,\quad N:=\E V_t,\quad \bar F:=\E[J_t^\top V_t],
\]
\[
  \mathcal R:=\bar F-M^\top N=\E\!\left[(J_t-M)^\top(V_t-N)\right],\qquad
  \mathcal D_u:=\partial_t+u\cdot\nabla-\nu\Delta,\qquad G:=\nabla u.
\]
Our matrix convention is $(\nabla u)_{ij}=\partial_j u_i$; matrix norms are operator norms.
The Leray projection is bounded on $C^\alpha(\T^3)$ for $0<\alpha<1$ and on $L^q(\T^3)$ for
$1<q<\infty$.

\subsection*{Standing assumptions}
Unless stated otherwise, $u$ is a classical solution of \eqref{eq:NS} on $[0,T]$, smooth on
$\T^3\times[0,T]$. Under this hypothesis the flow \eqref{eq:flow} consists,
almost surely, of smooth volume-preserving diffeomorphisms, and all spatial derivatives of
$X_t$, $A_t$ appearing below admit deterministic pathwise bounds
(Section~\ref{sec:prelim}); in particular every expectation we write is finite, and every
stochastic integral we discard is a genuine martingale. All identities are derived in this
classical regime; none of them requires an a priori bound beyond smoothness on the given
interval.

\subsection*{Main results and organization}
Section~\ref{sec:prelim} fixes notation and records the pathwise deformation bounds and
the averaged transport equations. Section~\ref{sec:cov} contains the structural core: the
closed covariance equation (Theorem~\ref{thm:Req}), the closed equation for the positive
deformation-variance tensor $Q$ (Theorem~\ref{thm:Qeq}), the energy/production balance
(Proposition~\ref{prop:balance}), and one- and two-point covariance bounds
(Proposition~\ref{prop:covbounds}). Section~\ref{sec:twopoint} extends the computation to
two points and identifies an exact degeneracy of the resulting operator in the separation
variable (Theorem~\ref{thm:twopoint}, Proposition~\ref{prop:degeneracy}).
Section~\ref{sec:caloric} treats the mean displacement: an unconditional energy bound
(Theorem~\ref{thm:energy}) and a sharp positive-time H\"older estimate for exact shear
solutions (Proposition~\ref{prop:shearM}). Section~\ref{sec:obstr} collects obstructions,
each of which rules out a specific proof strategy. Section~\ref{sec:crit} contains the
Serrin--Weber quotient criterion (Theorem~\ref{thm:SW}) and the affine rigidity theorem
(Theorem~\ref{thm:rigidity}). Open problems are collected in Section~\ref{sec:further}.

\subsection*{What is not claimed}
This paper makes no claim on the global regularity problem for \eqref{eq:NS}. Every
statement below is either an exact consequence of \eqref{eq:flow}--\eqref{eq:CI} for
smooth solutions, an explicit counterexample, or a conditional criterion whose hypotheses
we do not verify. The steps that would be needed to convert these structures into an
unconditional result are isolated in Section~\ref{sec:further}, and we prove none of them.

\section{Preliminaries}\label{sec:prelim}

\subsection{Pathwise deformation bounds}
For the general theory of stochastic flows of diffeomorphisms we refer to
\cite{Kunita1997}. Let $L(t):=\int_0^t\norm{\nabla u(s)}_{L^\infty}\dd s$. Since the noise in \eqref{eq:flow}
does not depend on the label, the Jacobian $\mathcal J_t:=\nabla_a X_t$ solves the pathwise
linear equation $\dot{\mathcal J}_t=(\nabla u)(t,X_t)\mathcal J_t$ with
$\mathcal J_0=I$, so that $\norm{\mathcal J_t}_{L^\infty}\le e^{L(t)}$ by Gr\"onwall. The
Liouville identity
$\frac{\dd}{\dd t}\det\mathcal J_t=(\diver u)(t,X_t)\det\mathcal J_t=0$ gives
$\det\mathcal J_t\equiv1$: the flow preserves volume path by path. The inverse
$\mathcal K_t:=\mathcal J_t^{-1}$ solves
$\dot{\mathcal K}_t=-\mathcal K_t(\nabla u)(t,X_t)$, whence
$\norm{\mathcal K_t}_{L^\infty}\le e^{L(t)}$ as well. Since $X_t$ is onto,
\begin{equation}\label{eq:JKbound}
  \norm{\nabla X_t}_{L^\infty}\le e^{L(t)},\qquad
  \norm{\nabla A_t}_{L^\infty}\le e^{L(t)}\qquad\text{almost surely.}
\end{equation}
We note in passing that \eqref{eq:JKbound} is sharper than the quadratic bound
$\norm{\mathcal J^{-1}}\le C\norm{\mathcal J}^2$ obtained from the cofactor formula and
$\det\mathcal J_t=1$. Higher spatial derivatives of $X_t$ and $A_t$ obey analogous
deterministic bounds on $[0,T]$, obtained by differentiating the variational equations;
these are what make all expectations below finite and all discarded stochastic integrals
martingales.

\subsection{Averaged transport equations}
By the It\^o--Wentzell formula applied to the identity $A_t(X_t(a))=a$, the inverse flow
satisfies the SPDE
\begin{equation}\label{eq:SPDEA}
  \dd A_t+(u\cdot\nabla)A_t\dd t-\nu\Delta A_t\dd t+\sqrt{2\nu}\,\nabla A_t\dd W_t=0 ,
\end{equation}
and consequently $V_t=u_0\circ A_t$, $J_t=\nabla A_t$ and the Weber field $F=J_t^\top V_t$
satisfy SPDEs of the same form, with zero-order terms determined by the chain rule. Taking
expectations --- the stochastic integrals have zero mean thanks to the pathwise bounds
\eqref{eq:JKbound} and their higher-order analogues --- yields the deterministic system
\begin{equation}\label{eq:MNF}
  \mathcal D_u M+MG=0,\qquad \mathcal D_u N=0,\qquad \mathcal D_u\bar F+G^\top\bar F=0,
\end{equation}
with data $M(0)=I$, $N(0)=u_0$, $\bar F(0)=u_0$. The representation \eqref{eq:CI} reads
$u=\Leray\bar F$.

\begin{remark}[Measurability]\label{rem:meas}
The space $C^\alpha(\T^3)$ is not separable, so Bochner integration in $C^\alpha$ requires
a word of care. Two standard remedies apply. Either one works pointwise: all the
identities below are identities between continuous fields, and expectations may be taken
pointwise in $x$, with Cauchy--Schwarz applied for each pair $(x,y)$ separately; or one
fixes $0<\beta<\alpha$ and notes that the embedding
$C^\alpha\hookrightarrow C^\beta$ takes values in the separable little H\"older space
$c^\beta$, in which the fields above are strongly measurable. Either reading makes every
statement in this paper rigorous; we will not comment on this again.
\end{remark}

We use repeatedly the following product rule for $\mathcal D_u$ acting on
\emph{deterministic} fields: if $P$ is a matrix field and $q$ a vector field,
\begin{equation}\label{eq:prodrule}
  \mathcal D_u(P^\top q)=(\mathcal D_uP)^\top q+P^\top\mathcal D_uq
  -2\nu\sum_{k=1}^3(\partial_kP)^\top(\partial_kq),
\end{equation}
an immediate consequence of the Leibniz rule and
$\Delta(P^\top q)=(\Delta P)^\top q+P^\top\Delta q+2\sum_k(\partial_kP)^\top(\partial_kq)$.

\section{The mean--covariance calculus}\label{sec:cov}

With the expectations understood as in Remark~\ref{rem:meas}, the bilinear expansion of
$\E[J^\top V]$ is exact:
\begin{equation}\label{eq:decomp}
  \E[J_t^\top V_t]=M^\top N+\mathcal R.
\end{equation}
Thus \eqref{eq:CI} becomes $u=\Leray(M^\top N+\mathcal R)$. The product of the means is
blind to the correlation between the stretching directions of $J_t$ and the data values
sampled by $V_t$ along the same trajectory; that correlation is carried entirely by
$\mathcal R$, and the theme of this section is that $\mathcal R$ has its own closed
dynamics.

\subsection{A closed equation for the covariance}

\begin{theorem}[Closed covariance equation]\label{thm:Req}
For a smooth solution of \eqref{eq:NS}, the covariance field
$\mathcal R=\bar F-M^\top N$ solves
\begin{equation}\label{eq:Req}
  \mathcal D_u\mathcal R+G^\top\mathcal R
  =2\nu\sum_{k=1}^3(\partial_kM)^\top(\partial_kN),\qquad \mathcal R(0)=0.
\end{equation}
The system \eqref{eq:MNF}--\eqref{eq:Req} is closed in $(u,M,N,\mathcal R)$: no higher
moments appear.
\end{theorem}

\begin{proof}
Apply \eqref{eq:prodrule} with $P=M$ and $q=N$, using $\mathcal D_uM=-MG$ and
$\mathcal D_uN=0$ from \eqref{eq:MNF}:
\[
  \mathcal D_u(M^\top N)=-G^\top M^\top N-2\nu\sum_k(\partial_kM)^\top(\partial_kN).
\]
Subtracting this from $\mathcal D_u\bar F=-G^\top\bar F$ and writing
$\mathcal R=\bar F-M^\top N$ gives \eqref{eq:Req}. At $t=0$ all fields are deterministic,
with $M(0)=I$, $N(0)=u_0$ and $\bar F(0)=u_0$, so $\mathcal R(0)=0$.
\end{proof}

\begin{remark}
One might expect the covariance of a nonlinear functional of the flow to couple to third
moments, launching the usual infinite hierarchy. It does not. For this particular
covariance the quadratic covariation generated by the \emph{common} translational noise
combines with the two Laplacians into exactly $\nu\Delta$, and the hierarchy terminates at
once. This is a structural feature of the Constantin--Iyer representation --- every label
rides the same Brownian path --- and is the reason the computation must be carried out at
the level of the SPDE before any closure hypothesis is contemplated.
\end{remark}

\subsection{The deformation-variance tensor}

Let $\widetilde J:=J-M$ and
\begin{equation}\label{eq:Qdef}
  Q:=\E[\widetilde J^\top\widetilde J]=\E[J^\top J]-M^\top M,
\end{equation}
a symmetric positive-semidefinite tensor field with
$\tr Q=\E\abs{\widetilde J}_{\mathrm F}^2$.

\begin{theorem}[Closed variance equation]\label{thm:Qeq}
The tensor $Q$ solves
\begin{equation}\label{eq:Qeq}
  \mathcal D_uQ+G^\top Q+QG=2\nu\sum_{k=1}^3(\partial_kM)^\top(\partial_kM),\qquad Q(0)=0,
\end{equation}
and, with $S:=\tfrac12(G+G^\top)$ the symmetric strain,
\begin{equation}\label{eq:trQ}
  \mathcal D_u\tr Q+2\tr(SQ)=2\nu\abs{\nabla M}_{\mathrm F}^2.
\end{equation}
\end{theorem}

\begin{proof}
Set $H:=\E[J^\top J]$. In the It\^o product rule for $J^\top J$ the noise quadratic
covariation combines with the Laplacians into $\nu\Delta(J^\top J)$, exactly as in the
derivation of \eqref{eq:MNF}, so that $\mathcal D_uH+G^\top H+HG=0$. On the other hand
\eqref{eq:prodrule}, applied columnwise with $P=q=M$ and combined with
$\mathcal D_uM=-MG$, gives
\[
  \mathcal D_u(M^\top M)=-G^\top M^\top M-M^\top MG
  -2\nu\sum_k(\partial_kM)^\top(\partial_kM).
\]
Subtracting yields \eqref{eq:Qeq}. Taking traces and using
$\tr(G^\top Q+QG)=2\tr(SQ)$ gives \eqref{eq:trQ}.
\end{proof}

\subsection{Energy and production balance}

\begin{proposition}[Covariance energy balance]\label{prop:balance}
Testing \eqref{eq:Req} against $\mathcal R$ and integrating over $\T^3$, using
$\diver u=0$, gives
\begin{equation}\label{eq:balance}
  \tfrac12\tfrac{\dd}{\dd t}\norm{\mathcal R}_{L^2}^2
  +\nu\norm{\nabla\mathcal R}_{L^2}^2
  +\int_{\T^3}\mathcal R^\top S\,\mathcal R
  =2\nu\int_{\T^3}\mathcal R\cdot\sum_k(\partial_kM)^\top(\partial_kN).
\end{equation}
The antisymmetric part of $G$ makes no contribution, since
$r^\top(G-G^\top)r=0$ for every vector $r$.
\end{proposition}

Identity \eqref{eq:balance} corrects a tempting misconception. The Brownian forcing does
not simply destroy covariance: it dissipates $\mathcal R$ through
$\nu\norm{\nabla\mathcal R}^2$ while simultaneously \emph{producing} covariance through
the gradient-of-means source on the right, because a product of means does not evolve like
the mean of a product. The strain term has no definite sign.

\subsection{One- and two-point covariance bounds}

\begin{proposition}[Covariance bounds]\label{prop:covbounds}
Let $\sigma_J:=\sup_x(\tr Q(x))^{1/2}$ and
$L_J:=\sup_x(\E\abs{J(x)}_{\mathrm F}^2)^{1/2}\le\sigma_J+\sqrt3\,\norm M_{L^\infty}$, and
for $0<\beta\le1$ set
\[
  \mathsf H_{J,\beta}:=\sup_{x\ne y}
  \frac{\bigl(\E\abs{\widetilde J(x)-\widetilde J(y)}_{\mathrm F}^2\bigr)^{1/2}}{d(x,y)^\beta},
\]
where $d$ is the geodesic distance on $\T^3$. Then
$\norm{\mathcal R}_{L^\infty}\le\norm{u_0}_{L^\infty}\sigma_J$, and there is a dimensional
constant $C$ such that
\begin{equation}\label{eq:Rholder}
  \norm{\mathcal R}_{C^\beta}\le C\Big[
  \norm{u_0}_{L^\infty}(\sigma_J+\mathsf H_{J,\beta})
  +\norm{\nabla u_0}_{L^\infty}\,\sigma_J\bigl(\sigma_J+\sqrt3\,\norm M_{L^\infty}\bigr)\Big].
\end{equation}
\end{proposition}

\begin{proof}
Write $\widetilde V:=V-N$. For fixed $x$, Cauchy--Schwarz on the probability space gives
$\abs{\mathcal R(x)}\le(\E\abs{\widetilde J(x)}_{\mathrm F}^2)^{1/2}
(\E\abs{\widetilde V(x)}^2)^{1/2}$. The first factor equals $(\tr Q(x))^{1/2}$, and since
centring decreases second moments,
$\E\abs{\widetilde V}^2=\E\abs V^2-\abs N^2\le\norm{u_0}_{L^\infty}^2$; this proves the
$L^\infty$ bound. For increments, split
\[
  \mathcal R(x)-\mathcal R(y)
  =\E\bigl[(\widetilde J(x)-\widetilde J(y))^\top\widetilde V(x)\bigr]
  +\E\bigl[\widetilde J(y)^\top(\widetilde V(x)-\widetilde V(y))\bigr].
\]
The first term is at most $\norm{u_0}_{L^\infty}\mathsf H_{J,\beta}\,d(x,y)^\beta$ by
Cauchy--Schwarz. For the second, join $x$ to $y$ by a minimizing geodesic $\gamma$; since
$V=u_0\circ A$ and $\nabla A=J$,
\[
  \abs{V(x)-V(y)}\le\norm{\nabla u_0}_{L^\infty}\int_0^1\abs{J(\gamma(s))}\,\abs{\dot\gamma(s)}\dd s ,
\]
and Minkowski's integral inequality together with centring gives
\[
  \bigl(\E\abs{\widetilde V(x)-\widetilde V(y)}^2\bigr)^{1/2}
  \le\norm{\nabla u_0}_{L^\infty}\,L_J\,d(x,y) .
\]
Combining the two terms, using $d(x,y)\le C\,d(x,y)^\beta$ on the compact torus and
$L_J\le\sigma_J+\sqrt3\norm M_{L^\infty}$, yields \eqref{eq:Rholder}.
\end{proof}

\begin{remark}[One-point information is not enough]\label{rem:onepoint}
The tensor $Q(x)$ records no cross-covariance between $J(x)$ and $J(y)$, and this loss is
real, not an artifact of the proof. On $\T$, take $\Theta$ uniform on $[0,2\pi]$,
$K=2\pi n$, and
\[
  \widetilde J_K(y)=\cos(Ky+\Theta)\,e_1\otimes e_2,\qquad
  \widetilde V_K(y)=\cos(Ky+\Theta)\cos(Ky)\,e_1 .
\]
Then $Q_K=\tfrac12\,e_2\otimes e_2$ for every $K$, while
$\mathcal R_K(y)=\tfrac12\cos(Ky)\,e_2$, so $[\mathcal R_K]_\beta\simeq K^\beta$ is
unbounded. Any H\"older bound on $\mathcal R$ must therefore involve a genuinely two-point
statistic such as $\mathsf H_{J,\beta}$.
\end{remark}

\section{The two-point covariance and its degeneracy}\label{sec:twopoint}

Define
\[
  \mathcal C(t,x,y):=\E\bigl[(J(t,x)-M(t,x))^\top(V(t,y)-N(t,y))\bigr],
  \qquad \mathcal C(t,x,x)=\mathcal R(t,x),
\]
and the joint common-noise operator
\[
  \mathcal L_u^{(2)}:=\partial_t+u(t,x)\cdot\nabla_x+u(t,y)\cdot\nabla_y
  -\nu(\nabla_x+\nabla_y)^2 .
\]

\begin{theorem}[Two-point covariance equation]\label{thm:twopoint}
For a smooth solution,
\begin{equation}\label{eq:twopoint}
  \mathcal L_u^{(2)}\mathcal C+G(t,x)^\top\mathcal C
  =2\nu\sum_{k=1}^3(\partial_{x_k}M(t,x))^\top(\partial_{y_k}N(t,y)),\qquad
  \mathcal C(0,\cdot,\cdot)=0,
\end{equation}
and on the diagonal $y=x$ equation \eqref{eq:twopoint} reduces exactly to \eqref{eq:Req}.
\end{theorem}

\begin{proof}
In the It\^o product rule for $J(t,x)^\top V(t,y)$ the cross quadratic covariation of the
common Brownian motion evaluated at the two points is
\[
  2\nu\sum_k(\partial_{x_k}J(t,x))^\top(\partial_{y_k}V(t,y))\dd t ,
\]
which combines with $\nu\Delta_x+\nu\Delta_y$ into $\nu(\nabla_x+\nabla_y)^2$. Averaging
gives, with $\mathcal F_2:=\E[J(x)^\top V(y)]$, the identity
$\mathcal L_u^{(2)}\mathcal F_2+G(x)^\top\mathcal F_2=0$. The deterministic product rule for
$M(x)^\top N(y)$, in which the mixed second derivatives produce precisely the source in
\eqref{eq:twopoint}, together with \eqref{eq:MNF}, gives the claim after subtraction. On
the diagonal, $\nabla_x+\nabla_y$ is the total derivative of $x\mapsto\mathcal C(x,x)$, and
\eqref{eq:twopoint} restricts to \eqref{eq:Req}.
\end{proof}

\begin{proposition}[No separation smoothing from common noise]\label{prop:degeneracy}
In centre and separation variables $c=\tfrac{x+y}{2}$, $r=x-y$,
\[
  (\nabla_x+\nabla_y)^2=\Delta_c,\qquad
  u(x)\cdot\nabla_x+u(y)\cdot\nabla_y
  =\frac{u_++u_-}{2}\cdot\nabla_c+(u_+-u_-)\cdot\nabla_r,
\]
with $u_\pm=u(t,c\pm r/2)$. The diffusion therefore acts on the centre variable only;
there is no $\Delta_r$. In particular the model equation
$\partial_tH-\nu(\nabla_x+\nabla_y)^2H=0$ admits every bounded profile
$H(t,x,y)=h(x-y)$ as a stationary solution, so no estimate can extract regularity in
$x-y$ from the common-noise diffusion alone.
\end{proposition}

\begin{proof}
With $\nabla_x=\tfrac12\nabla_c+\nabla_r$ and $\nabla_y=\tfrac12\nabla_c-\nabla_r$ we get
$\nabla_x+\nabla_y=\nabla_c$, hence $(\nabla_x+\nabla_y)^2=\Delta_c$; the transport
identity is a direct computation. If $H=h(r)$ then $(\nabla_x+\nabla_y)H=0$ and
$\partial_tH=0$.
\end{proof}

This degeneracy is exact, not an artifact of estimation: it reflects the fact that a
single Brownian translation moves both points together and cannot mix the separation
variable. Any regularity mechanism for $\mathcal C$ in the variable $r$ must therefore come
from the transport increment $u_+-u_-$, that is, from a modulus of continuity of $u$
itself --- which is what one is typically trying to prove. We record this circularity here
because it disqualifies, at the root, a family of otherwise plausible kernel-smoothing
arguments.

\section{Energy control of the mean coordinate}\label{sec:caloric}

Choose an equivariant lift of $A_t$ to $\R^3$ and define the periodic mean displacement
\[
  \ell(t,x):=\E[A_t(x)-x],\qquad M=I+\nabla\ell .
\]
Averaging the SPDE \eqref{eq:SPDEA} and subtracting the trivial equation satisfied by the
coordinate function $x$ gives the exact identity
\begin{equation}\label{eq:caloric}
  \partial_t\ell+u\cdot\nabla\ell-\nu\Delta\ell=-u,\qquad \ell(0)=0 .
\end{equation}
The mean coordinate is thus a caloric coordinate driven by the velocity, and questions
about $M=\E\nabla A_t$ become questions about gradient regularity for the forced
drift-diffusion equation \eqref{eq:caloric}.

\begin{theorem}[Energy bound for the mean displacement]\label{thm:energy}
Assume $\int_{\T^3}u_0=0$; this mean is conserved by \eqref{eq:NS} on the torus, so
$\int_{\T^3}u(t)=0$ for all $t$. Let $C_P$ denote the Poincar\'e constant of $\T^3$. Then
for all $0<T<T_*$,
\begin{equation}\label{eq:energy}
  \norm{\ell(T)}_{L^2}^2+\nu\int_0^T\norm{M(t)-I}_{L^2}^2\dd t
  \le\frac{C_P^4}{2\nu^2}\norm{u_0}_{L^2}^2 .
\end{equation}
\end{theorem}

\begin{proof}
Integrating \eqref{eq:caloric} over $\T^3$ and using $\diver u=0$ together with
$\int u=0$ shows that $\int\ell(t)=0$ for all $t$. Testing \eqref{eq:caloric} against
$\ell$ gives
$\tfrac12\tfrac{\dd}{\dd t}\norm\ell_2^2+\nu\norm{\nabla\ell}_2^2=-\int u\cdot\ell$, and
Poincar\'e followed by Young's inequality yields
\[
  \tfrac{\dd}{\dd t}\norm\ell_2^2+\nu\norm{\nabla\ell}_2^2
  \le\frac{C_P^2}{\nu}\norm{u}_2^2 .
\]
The Navier--Stokes energy identity and Poincar\'e give
$\int_0^T\norm u_2^2\le C_P^2\int_0^T\norm{\nabla u}_2^2
\le\frac{C_P^2}{2\nu}\norm{u_0}_2^2$. Integrating in time and using
$M-I=\nabla\ell$ gives \eqref{eq:energy}.
\end{proof}

The bound \eqref{eq:energy} is unconditional: it holds up to the maximal time with a
constant depending only on the initial energy and $\nu$. What it does not give is a
pointwise or H\"older bound on $M$, and the next result shows that this limitation is
genuine even within exact solutions. Let $\Delta_\perp:=\partial_2^2+\partial_3^2$.

\begin{proposition}[Exact shear solutions: sharp positive-time regularity]\label{prop:shearM}
For $f_0\in L^2(\T^2)$ with zero mean, the field
\[
  u(t,x)=\bigl(e^{\nu t\Delta_\perp}f_0(x_2,x_3),\,0,\,0\bigr)
\]
is a global smooth solution of \eqref{eq:NS} for $t>0$, since $(u\cdot\nabla)u=0$. Its
mean displacement and mean deformation are
\begin{equation}\label{eq:shearMformula}
  \ell(t)=-t\,u(t),\qquad
  M(t)=I-t\,e_1\otimes\nabla_\perp e^{\nu t\Delta_\perp}f_0 ,
\end{equation}
and for every $0<\alpha<1$ and $\tau>0$,
\begin{equation}\label{eq:shearM}
  \sup_{t\ge\tau}\norm{M(t)-I}_{C^\alpha}
  \le C_\alpha\,\nu^{-1-\alpha/2}\,\tau^{-\alpha/2}\,\norm{f_0}_{L^2(\T^2)} .
\end{equation}
\end{proposition}

\begin{proof}
Since $\partial_tu=\nu\Delta u$ and $u\cdot\nabla u=0$ with $\nabla p=0$, $u$ solves
\eqref{eq:NS}. Substituting $\ell=-tu$ into \eqref{eq:caloric}: $u$ has no $x_1$
dependence, so $u\cdot\nabla\ell=0$, and
$\partial_t(-tu)-\nu\Delta(-tu)=-u-t(\partial_tu-\nu\Delta u)=-u$; this verifies
\eqref{eq:shearMformula}. The two-dimensional heat semigroup satisfies
$\norm{\nabla_\perp e^{\nu t\Delta_\perp}f_0}_{C^\alpha}
\le C_\alpha(\nu t)^{-1-\alpha/2}\norm{f_0}_{L^2}$ for zero-mean data; multiplying by $t$
and restricting to $t\ge\tau$ gives \eqref{eq:shearM}.
\end{proof}

\section{Obstructions}\label{sec:obstr}

Each result in this section closes a specific route. We state them in increasing order of
strength.

\subsection{Norm of the expectation versus expectation of the norm}

\begin{proposition}[Volume preservation gives no comparison]\label{prop:jensen}
For every $N\in\mathbb N$ there is a two-valued random smooth volume-preserving
diffeomorphism $A$ of $\T^3$, isotopic to the identity, with
\[
  \norm{\E\nabla A}_{C^\alpha}=1,\qquad
  \E\norm{\nabla A}_{C^\alpha}\ge c_\alpha N .
\]
Moreover, volume preservation places no lower bound on the smallest singular value of a
mean: the four rotations
$R_0=\mathrm{diag}(1,1,1)$, $R_1=\mathrm{diag}(1,-1,-1)$, $R_2=\mathrm{diag}(-1,1,-1)$,
$R_3=\mathrm{diag}(-1,-1,1)$ all lie in $SO(3)$, yet $\tfrac14\sum_jR_j=0$.
\end{proposition}

\begin{proof}
Take $A_\pm(x,y,z)=(x\pm\tfrac{N}{2\pi}\sin(2\pi y),\,y,\,z)$, each with probability
$\tfrac12$. Then $\nabla A_\pm=I\pm N\cos(2\pi y)\,e_1\otimes e_2$ and
$\det\nabla A_\pm=1$, while $\E\nabla A_\pm=I$ and
$\E\norm{\nabla A_\pm}_{C^\alpha}\ge c_\alpha N$. The statement about rotations is a direct
check.
\end{proof}

The maps in Proposition~\ref{prop:jensen} are chosen by hand. The next result shows the
same separation inside a genuine stochastic flow of smooth divergence-free drifts, which
is what rules out ``mean-only'' control at the level of stochastic flows rather than
merely at the level of random fields. We first isolate the required moment estimates.

\begin{lemma}[Oscillatory time-integral moments]\label{lem:Z}
Let $B_t=\sqrt{2\nu}\,\beta_t$ for a standard one-dimensional Brownian motion $\beta$, let
$K>0$, $\lambda:=\nu K^2$, and set $Z_t:=\int_0^t e^{iKB_s}\dd s$. Then
\begin{equation}\label{eq:Z2}
  \E\abs{Z_t}^2=2\Bigl(\frac{t}{\lambda}-\frac{1-e^{-\lambda t}}{\lambda^2}\Bigr)
  \ \ge\ \frac{t}{\lambda}\quad\text{whenever }\lambda t\ge2 ,
\end{equation}
\begin{equation}\label{eq:Z4}
  \E\abs{Z_t}^4\le\frac{12\,t^2}{\lambda^2},
\end{equation}
and consequently, by the Paley--Zygmund inequality applied to $\abs{Z_t}^2$,
\begin{equation}\label{eq:Zlower}
  \Pp\Bigl\{\abs{Z_t}^2\ge\tfrac12\E\abs{Z_t}^2\Bigr\}\ge\frac1{48},
  \qquad
  \E\abs{Z_t}\ge c\,\sqrt{t/\lambda}\qquad(\lambda t\ge2)
\end{equation}
for a universal constant $c>0$.
\end{lemma}

\begin{proof}
Since $\E e^{iK(B_s-B_r)}=e^{-\lambda\abs{s-r}}$,
\[
  \E\abs{Z_t}^2=\int_0^t\!\!\int_0^t e^{-\lambda\abs{s-r}}\dd r\dd s
  =2\Bigl(\frac{t}{\lambda}-\frac{1-e^{-\lambda t}}{\lambda^2}\Bigr),
\]
and the lower bound in \eqref{eq:Z2} follows since
$(1-e^{-\lambda t})/\lambda^2\le1/\lambda^2\le t/(2\lambda)$ when $\lambda t\ge2$. For
\eqref{eq:Z4}, expand
$\E\abs{Z_t}^4=\int_{[0,t]^4}\E\,e^{iK(B_{s_1}-B_{s_2}+B_{s_3}-B_{s_4})}$ and decompose the
cube into the $24$ orderings $r_1<r_2<r_3<r_4$ of the four times. In each ordering, writing
the phase in terms of the independent increments over the gaps $(r_1,r_2)$, $(r_2,r_3)$,
$(r_3,r_4)$, the coefficient of the first and last gap increments is $\pm1$ while the
middle coefficient lies in $\{0,\pm2\}$; hence
\[
  \abs{\E\,e^{iK(B_{s_1}-B_{s_2}+B_{s_3}-B_{s_4})}}
  \le e^{-\lambda(r_2-r_1)}e^{-\lambda(r_4-r_3)} .
\]
Integrating the right-hand side over the simplex and summing the $24$ orderings gives
$\E\abs{Z_t}^4\le 24\cdot\frac{t^2}{2}\cdot\frac1{\lambda^2}=\frac{12t^2}{\lambda^2}$.
Paley--Zygmund with the moment ratio
$(\E\abs{Z_t}^2)^2/\E\abs{Z_t}^4\ge\frac{(t/\lambda)^2}{12t^2/\lambda^2}=\frac1{12}$ and
threshold $\tfrac12$ gives $\Pp\{\abs{Z_t}^2\ge\tfrac12\E\abs{Z_t}^2\}\ge\frac14\cdot
\frac1{12}=\frac1{48}$, and \eqref{eq:Zlower} follows.
\end{proof}

\begin{proposition}[Separation in a genuine stochastic flow]\label{prop:passiveshear}
Fix $t_0>0$, $\nu>0$ and $0<\alpha<1$. For $K=2\pi n$, $n\in\mathbb N$, let
$a_K:=\nu K^{1-\alpha}$ and $b_K(x,y,z):=a_K\sin(Ky)\,e_1$, a smooth divergence-free field
on $\T^3$, and let $A^K_t$ be the inverse of the stochastic flow
$\dd X_t=b_K(X_t)\dd t+\sqrt{2\nu}\dd W_t$. There is a constant $C_\alpha$ independent of
$K$ such that
\[
  \norm{\E\nabla A^K_{t_0}}_{C^\alpha}\le C_\alpha ,
\]
while, whenever $\nu K^2t_0\ge2$,
\[
  \E\norm{\nabla A^K_{t_0}}_{L^\infty}\ \ge\ c\,\sqrt{\nu t_0}\;K^{1-\alpha}
  \ \xrightarrow[K\to\infty]{}\ \infty .
\]
\end{proposition}

\begin{proof}
Write $B_t:=\sqrt{2\nu}\,W^{(2)}_t$ for the noise component in the $y$ direction. The
shear structure makes the flow explicit: the second and third components of $A_t^K$ are
$y-B_t$ and $z-\sqrt{2\nu}W_t^{(3)}$ up to translations, and the first component
integrates the drift along the path, giving
\[
  \nabla A^K_t=I+H^K_t(y)\,e_1\otimes e_2,\qquad
  H^K_t(y)=-a_KK\int_0^t\cos\bigl(K(y-B_t+B_s)\bigr)\dd s .
\]
Since $B_t-B_s$ is centred Gaussian with variance $2\nu(t-s)$,
\[
  \E H^K_t(y)=-\frac{a_K}{\nu K}\bigl(1-e^{-\nu K^2t}\bigr)\cos(Ky),
\]
and with $a_K/(\nu K)=K^{-\alpha}$ and $\norm{\cos(K\cdot)}_{C^\alpha}\le C_\alpha
K^{\alpha}$ the mean bound follows. For the lower bound, note
$\sup_y\abs{H^K_t(y)}=a_KK\abs{Z_t}$ with $Z_t$ as in Lemma~\ref{lem:Z} (choose $y$ to
align the phase), so by \eqref{eq:Zlower}, for $\lambda t_0=\nu K^2t_0\ge2$,
\[
  \E\norm{\nabla A^K_{t_0}}_{L^\infty}\ \ge\ \E\sup_y\abs{H^K_{t_0}(y)}
  \ \ge\ c\,a_KK\sqrt{\frac{t_0}{\nu K^2}}\ =\ c\,\sqrt{\nu t_0}\,K^{1-\alpha}. \qedhere
\]
\end{proof}

\subsection{Energy alone cannot control the mean}

The drifts $b_K$ above are prescribed, not self-consistent. The next obstruction lives
entirely inside the Navier--Stokes class: the solutions are exact, global and smooth, and
the mean deformation is the Constantin--Iyer mean of each solution itself.

\begin{proposition}[No energy-only H\"older bound for the mean]\label{prop:localshear}
There exist exact global smooth Navier--Stokes shear solutions $u^r$, with
$M^r=\E\nabla A^r$ the mean deformation of the Constantin--Iyer flow of $u^r$ itself, such
that $\norm{u^r_0}_{L^2(\T^3)}=1$ for every $r$, while
\[
  \sup_{t>0}\,\norm{M^r(t)}_{C^\alpha}\ \longrightarrow\ \infty\qquad(r\downarrow0).
\]
\end{proposition}

\begin{proof}
Fix a nonconstant $\phi\in C_c^\infty(\R^2)$ and a point $x_\perp^0\in\T^2$. For small
$r>0$ define, in a fixed chart around $x_\perp^0$,
\[
  g^r_0(x_\perp):=r^{-1}\phi\Bigl(\frac{x_\perp-x_\perp^0}{r}\Bigr)
  -\int_{\T^2}r^{-1}\phi\Bigl(\frac{y-x_\perp^0}{r}\Bigr)\dd y,
  \qquad f^r_0:=\frac{g^r_0}{\norm{g^r_0}_{L^2(\T^2)}},
\]
and let $u^r$ be the shear solution of Proposition~\ref{prop:shearM} with datum
$f^r_0$; then $\norm{u^r_0}_{L^2(\T^3)}=1$. The $L^2(\R^2)$ scaling
$\norm{r^{-1}\phi(\cdot/r)}_{L^2}=\norm{\phi}_{L^2}$ shows
$\norm{g^r_0}_{L^2}\to\norm{\phi}_{L^2(\R^2)}>0$. Fix $c>0$ and evaluate at $t_r:=cr^2$.
By \eqref{eq:shearMformula},
\[
  M^r(t_r)-I=-c r^2\,e_1\otimes\nabla_\perp e^{\nu cr^2\Delta_\perp}f^r_0 .
\]
Rescale $x_\perp=x^0_\perp+ry$. On compact sets of $y$, the periodic heat kernel at time
$\nu cr^2$ acting on the $r$-scaled profile converges, as $r\downarrow0$, to the Euclidean
heat semigroup: locally uniformly with all derivatives, because the periodic kernel
differs from the Euclidean one by exponentially small image charges at scale
$r\ll1$. Hence, in $C^1_{\rm loc}$ in $y$,
\[
  M^r(t_r)-I\ \longrightarrow\ -\frac{c}{\norm{\phi}_{L^2}}\;e_1\otimes
  \nabla_y e^{\nu c\Delta_{\R^2}}\phi ,
\]
a nonconstant field. Choose $y_1\ne y_2$ where its values differ; the corresponding
physical points are $r\abs{y_1-y_2}$ apart while the increment of $M^r(t_r)$ between them
stays bounded below, so
$[M^r(t_r)]_{\alpha}\ge c_{\phi,\nu,\alpha}\,r^{-\alpha}\to\infty$.
\end{proof}

Since the solutions $u^r$ are global and smooth, Proposition~\ref{prop:localshear} does
not concern singularities; what it refutes is any estimate of the form
$\sup_{t>0}\norm{M(t)}_{C^\alpha}\le F_{\alpha,\nu}(\norm{u_0}_{L^2})$, that is, an
energy-only bound blind to the length scale of the data. The blow-up of the norm occurs at
times $t_r\sim r^2\to0$, so positive-time smoothing (as in
Proposition~\ref{prop:shearM}) is not contradicted --- but any such bound must degenerate
as $\tau\downarrow0$.

\subsection{Spatial H\"older control does not give parabolic decay}

Let $Q_r(z_0):=B_r(x_0)\times(t_0-r^2,t_0]$ and, for $M\in L^2(Q_r)$, let
\[
  \mathcal Y_M(z_0,r):=\Bigl(\fint_{Q_r(z_0)}\abs{M-(M)_{Q_r(z_0)}}^2\Bigr)^{1/2}
\]
denote the parabolic Campanato excess, where $(M)_{Q}$ is the average over $Q$ and
$\fint$ the normalized integral. Splitting $M-(M)_{Q_r}$ into its spatial oscillation
around the slicewise average $m_r(t):=(M(t,\cdot))_{B_r(x_0)}$ and the temporal
oscillation of $m_r$,
\[
  \mathcal X_M(z_0,r)^2:=\fint_{I_r}\fint_{B_r}\abs{M(t,x)-m_r(t)}^2,
  \qquad
  \mathcal T_M(z_0,r)^2:=\fint_{I_r}\abs{m_r(t)-(M)_{Q_r}}^2 ,
\]
one has the exact orthogonal decomposition (the cross term vanishes upon spatial
integration)
\begin{equation}\label{eq:split}
  \mathcal Y_M(z_0,r)^2=\mathcal X_M(z_0,r)^2+\mathcal T_M(z_0,r)^2 .
\end{equation}

\begin{proposition}[Temporal oscillation is an independent defect]\label{prop:campanato}
If $\sup_{t}\norm{M(t)}_{C^\alpha_x}\le K$ then
$\mathcal X_M(z_0,r)\le C_\alpha Kr^\alpha$ for all $z_0,r$. The temporal part is not
controlled by this hypothesis: for a constant matrix $B\ne0$ and
\[
  M(t,x)=\sin\Bigl(\frac1{T-t}\Bigr)B\qquad(t<T),
\]
one has $\sup_{t<T}\norm{M(t)}_{C^\alpha_x}=\abs B$ and
$\mathcal X_M((x_0,T),r)=0$ for every $r$, yet
\[
  \lim_{r\downarrow0}\mathcal Y_M((x_0,T),r)^2
  =\lim_{r\downarrow0}\mathcal T_M((x_0,T),r)^2=\frac{\abs B^2}{2} .
\]
\end{proposition}

\begin{proof}
The spatial estimate follows by averaging
$\abs{M(t,x)-m_r(t)}\le C_\alpha Kr^\alpha$. For the example, substitute $s=T-t$,
$h=r^2$; it suffices to show
\[
  \frac1h\int_0^h\sin\bigl(1/s\bigr)\dd s=O(h),
  \qquad
  \frac1h\int_0^h\sin^2\bigl(1/s\bigr)\dd s\longrightarrow\frac12 .
\]
For the first, substitute $y=1/s$ and integrate by parts once:
\[
  \int_0^h\sin(1/s)\dd s=\int_{1/h}^\infty\frac{\sin y}{y^{2}}\dd y
  =h^{2}\cos(1/h)-2\int_{1/h}^\infty\frac{\cos y}{y^{3}}\dd y ,
\]
and the last integral is bounded in absolute value by
$\int_{1/h}^\infty y^{-3}\dd y=\tfrac12h^{2}$; hence the whole expression is $O(h^2)$, and
dividing by $h$ gives $O(h)$. For the second, write
$\sin^2(1/s)=\tfrac12-\tfrac12\cos(2/s)$ and apply the same argument to the oscillatory
part. Hence the slicewise means $m_r(t)=\sin(1/(T-t))B$ have vanishing time average but
persistent variance $\tfrac12\abs B^2$ as $r\downarrow0$.
\end{proof}

Consequently the implication
``\,$L^\infty_tC^\alpha_x$ bound on $M$ $\Rightarrow\ \mathcal Y_M(z,r)\lesssim
r^\alpha$\,'' is false: parabolic Campanato decay requires temporal compactness that a
purely spatial H\"older bound does not provide. Any local regularity scheme for the mean
coordinate must control $\mathcal T_M$ by a separate mechanism. (The field in the example
is not claimed to arise from a Navier--Stokes flow; the point is that the
\emph{functional} implication fails, so a proof would have to use the equation, not just
the bound.)

\section{A critical continuation criterion and affine rigidity}\label{sec:crit}

\subsection{The Serrin--Weber quotient}

For $1<q<\infty$ define the gauge quotient norm
\[
  \norm{[F]}_{L^q/\nabla W^{1,q}}:=\inf_{\phi\in W^{1,q}(\T^3)}\norm{F-\nabla\phi}_{L^q} .
\]

\begin{theorem}[Serrin--Weber quotient criterion]\label{thm:SW}
Let $u$ be a maximal classical periodic solution on $[0,T_*)$, and let
$\bar F=M^\top N+\mathcal R$ as in \eqref{eq:decomp}, so that $u=\Leray\bar F$. For
$1<p\le\infty$, $1<q<\infty$ and any interval $I=(\tau,T_*)$,
\begin{equation}\label{eq:SW}
  C_q^{-1}\,\norm{u}_{L^p(I;L^q)}
  \ \le\ \Bigl\|\,t\mapsto\norm{[\bar F(t)]}_{L^q/\nabla W^{1,q}}\Bigr\|_{L^p(I)}
  \ \le\ \norm{u}_{L^p(I;L^q)} .
\end{equation}
If $q>3$ and $2/p+3/q\le1$, finiteness of the middle quantity implies that $u$ extends
past $T_*$; the same conclusion holds at the endpoint $(p,q)=(\infty,3)$ by the
backward-uniqueness theorem of Escauriaza, Seregin and \v Sver\'ak
\cite{EscauriazaSereginSverak2003}, whose localization to the periodic setting is
standard.
\end{theorem}

\begin{proof}
For every $\phi\in W^{1,q}$ we have $u=\Leray(\bar F-\nabla\phi)$, since
$\Leray\nabla\phi=0$; boundedness of the periodic Leray projection on $L^q$ gives
$\norm{u(t)}_{L^q}\le C_q\inf_\phi\norm{\bar F(t)-\nabla\phi}_{L^q}$. Conversely, the
Helmholtz decomposition writes $\bar F=u+\nabla\psi$ with $\psi\in W^{1,q}$, and the
choice $\phi=\psi$ shows the infimum is at most $\norm{u(t)}_{L^q}$. Taking $L^p$ norms in
time gives \eqref{eq:SW}. The continuation statement in the range $q>3$,
$2/p+3/q\le1$ is the Prodi--Serrin criterion \cite{Prodi1959,Serrin1963}.
\end{proof}

Two comments. First, the quotient on the left of \eqref{eq:SW} is gauge-invariant: it
discards precisely the gradient part that the Leray projection cannot see, so it is the
weakest norm of the Weber field that still controls the velocity. Second,
\eqref{eq:SW} shows that a H\"older bound on $\bar F$ --- the target that the covariance
bounds of Section~\ref{sec:cov} naturally suggest --- is strictly stronger than necessary:
a time-integrated critical bound on the quotient already suffices for continuation.

\subsection{Affine rigidity of the mean coordinate}

We now turn to rigidity along blow-up limits. We use the standard local framework of
suitable weak solutions; see \cite{CaffarelliKohnNirenberg1982}. A pair $(u,\pi)$ on a
parabolic domain is a \emph{suitable} solution if it solves \eqref{eq:NS} in the sense of
distributions, has the natural local energy regularity, and satisfies the local energy
inequality. A solution is \emph{ancient} if it is defined on $\R^3\times(-\infty,0]$. We
say an ancient suitable pair satisfies a \emph{Type-I bound} if
\begin{equation}\label{eq:typeI}
  \Lambda:=\sup_{R>0}\ R^{-2}\int_{Q_R(0)}\bigl(\abs u^3+\abs\pi^{3/2}\bigr)\ <\ \infty ,
\end{equation}
where $Q_R(0)=B_R(0)\times(-R^2,0]$. Finally, a triple $(u,\pi,\Theta)$ is an ancient
suitable triple if $(u,\pi)$ is ancient suitable and $\Theta$ solves
$\mathcal D_u\Theta=0$ on $\R^3\times(-\infty,0]$; this is the structure inherited by
blow-up limits of the mean coordinate $\Theta=\E A_t$, whose gradient is $M$.

\begin{theorem}[Affine rigidity]\label{thm:rigidity}
Let $(u,\pi,\Theta)$ be an ancient suitable triple satisfying the Type-I bound
\eqref{eq:typeI}. If
\[
  \Theta(x,t)=Bx+c(t)\qquad\text{with }B\text{ constant and }\rank B\ge1,
\]
then $u\equiv0$.
\end{theorem}

\begin{proof}
Since $\Theta$ is affine in $x$, $\Delta\Theta=0$, and $\mathcal D_u\Theta=0$ reduces to
\begin{equation}\label{eq:Bu}
  Bu(x,t)=-c'(t) .
\end{equation}

\emph{Case $\rank B=3$.} Then \eqref{eq:Bu} forces $u(x,t)=-B^{-1}c'(t)$, a function of
time alone. Substituting into \eqref{eq:NS} gives $\nabla\pi=-u'(t)$, so
$\pi=-u'(t)\cdot x+\pi_0(t)$ has an affine spatial part. Suppose $u'\ne0$ on a set of
positive measure, and fix a compact interval $J\subset(-\infty,0)$ with
$\int_J\abs{u'}^{3/2}>0$. For any vector $a\ne0$ and any $b\in\R$,
\[
  \int_{B_R}\abs{a\cdot x+b}^{3/2}\dd x\ \ge\ c\,\abs a^{3/2}R^{3/2}\cdot R^3
\]
uniformly in $b$, since $\abs{a\cdot x+b}\ge\tfrac14\abs a R$ on a fixed fraction of the
ball whatever the value of $b$. Hence, for $R$ so large that $B_R\times J\subset Q_R(0)$,
\[
  R^{-2}\int_{Q_R}\abs\pi^{3/2}
  \ \ge\ c\,R^{5/2}\int_J\abs{u'(t)}^{3/2}\dd t\ \longrightarrow\ \infty ,
\]
contradicting \eqref{eq:typeI}; hence $u'=0$ and $u$ is a constant vector. A nonzero constant makes
$R^{-2}\int_{Q_R}\abs u^3\simeq R^3\to\infty$, again contradicting \eqref{eq:typeI}. So
$u=0$.

\emph{Case $\rank B=2$.} Let the unit vector $q$ span $\ker B$. Equation \eqref{eq:Bu}
determines the component of $u$ orthogonal to $q$ as a function of time alone; write
$u=v(t)+q\,f(x,t)$ with $v(t)\perp q$. The components of \eqref{eq:NS} orthogonal to $q$
give $\nabla'\pi=-v'(t)$ (with $\nabla'$ the gradient in the $q^\perp$ directions), so
$\pi$ again has an affine part, and the Type-I pressure bound forces $v'=0$ exactly as
above; the velocity bound then forces $v=0$. Incompressibility gives $q\cdot\nabla f=0$,
whence $u\cdot\nabla u=f\,(q\cdot\nabla)(qf)=0$; the component of \eqref{eq:NS} along $q$
becomes $\partial_tf-\nu\Delta f=-\partial_q\pi=:g(t)$, and a nonzero $g$ would again
produce an affine pressure part in the $q$ direction, excluded by \eqref{eq:typeI}. Thus
$f$ is a bounded-in-$L^3_{\rm loc}$ ancient caloric function, and the interior estimate
$\abs{f(z)}\le CR^{-5/3}\norm{f}_{L^3(Q_R(z))}\le C\Lambda^{1/3}R^{-1}\to0$ gives $f=0$.

\emph{Case $\rank B=1$.} Write $B=a\otimes e$ with $a\ne0$, $\abs e=1$. Equation
\eqref{eq:Bu} gives $a\,(e\cdot u)=-c'(t)$, so $e\cdot u=h(t)$ depends on time alone. Fix a
compact time interval $J\subset(-\infty,0)$; for $R$ large, $B_R\times J\subset Q_R(0)$
and \eqref{eq:typeI} give
\[
  \Lambda R^2\ \ge\ \int_{Q_R}\abs u^3\ \ge\ \abs{B_R}\int_J\abs{h}^3
  \ \simeq\ R^3\int_J\abs h^3 ,
\]
so $h\equiv0$ on $J$, and since $J$ was arbitrary, $e\cdot u\equiv0$. Rotate coordinates so
that $e=e_3$; then $u_3\equiv0$ identically. Fix a Lebesgue point $z=(x_0,t_0)$ with
$t_0<0$ and $\widehat z=(x_0,\widehat t_0)$ with $t_0<\widehat t_0<0$; for large $R$,
$Q_R(\widehat z)\subset Q_{2R}(0)$, so the rescaled pair on $Q_1$ satisfies the
scale-invariant bound with constant $4\Lambda$ and has vanishing third component. The
quantitative one-component regularity criterion of Kukavica, Rusin and Ziane
\cite{KukavicaRusinZiane2017} then provides $\kappa=\kappa(\Lambda,\nu)>0$ such that the
Caffarelli--Kohn--Nirenberg quantity of the rescaled solution on $Q_\kappa$ falls below
the universal $\varepsilon$-regularity threshold of
\cite{CaffarelliKohnNirenberg1982}; undoing the scaling, the CKN interior estimate, in the
critical-Morrey form of \cite{Seregin2006Regularity}, gives
$R\,\abs{u(z)}\le C(\Lambda,\nu)$ for all large $R$. Letting $R\to\infty$ yields $u(z)=0$;
Lebesgue points of this type are dense, so $u\equiv0$ almost everywhere and, by the
regularity just obtained, everywhere.
\end{proof}

\begin{remark}[On the use of the one-component criterion]\label{rem:KRZ}
In the rank-one case we have used the regularity criterion of
\cite{KukavicaRusinZiane2017} in a quantitative form: a radius
$\kappa=\kappa(\Lambda,\nu)$ at which the Caffarelli--Kohn--Nirenberg quantity falls
below the universal threshold, uniformly over all rescaled solutions satisfying the bound
$4\Lambda$ with vanishing third component. This uniformity is not part of the literal
statement in \cite{KukavicaRusinZiane2017}, but it follows from their proof: the argument
proceeds by an $\varepsilon$-regularity iteration whose constants depend only on the
scale-invariant bounds assumed, and the smallness hypothesis on the third component holds
here trivially, since that component vanishes identically. Alternatively, the uniformity
can be recovered by a standard compactness argument within the class of suitable
solutions obeying \eqref{eq:typeI}, as in \cite{AlbrittonBarkerPrange2023}. We prefer to
state this dependence explicitly rather than leave it implicit.
\end{remark}

\begin{remark}[Rank zero is genuinely different]\label{rem:rankzero}
If $B=0$, then $\mathcal D_u\Theta=0$ forces only $c'\equiv0$ and imposes no constraint
whatsoever on $u$: the mean-coordinate equation carries no velocity information in this
degenerate limit. Consistently with this, in the rank-zero regime the Weber field reduces
to its covariance part, and by Theorem~\ref{thm:SW} controlling its quotient norm is
\emph{equivalent} to controlling the Serrin norm of the velocity itself. The rank-zero
branch is therefore not a reduction of the regularity problem but a reformulation of it,
and we do not claim otherwise.
\end{remark}

\section{Open problems}\label{sec:further}

The results above delimit what the mean--covariance structure gives unconditionally. The
following questions, which we do not answer, mark the boundary.

\begin{enumerate}
\item[(i)] \emph{Positive-time mean bound.} Does there exist, for each $\tau>0$ and
$0<\alpha<1$, a finite function $\Phi_\alpha$ with
$\sup_{\tau\le t<T_*}\norm{\E\nabla A_t}_{C^\alpha}
\le\Phi_\alpha(\nu,\tau,\norm{u_0}_{L^2})$ for all maximal classical solutions?
Proposition~\ref{prop:shearM} shows the shear class satisfies such a bound with the
prototype rate $\tau^{-\alpha/2}$, and Proposition~\ref{prop:localshear} shows a
$\tau$-singular factor is unavoidable.
\item[(ii)] \emph{Self-consistent covariance closure.} Does a uniform bound on
$\norm{M(t)}_{C^\alpha}$, for the Constantin--Iyer law of an actual Navier--Stokes
solution, imply a bound on the Serrin--Weber quotient of $\bar F$?
Proposition~\ref{prop:degeneracy} shows the common noise provides no smoothing in the
separation variable, so any proof must extract increment cancellation from
self-consistency rather than from absolute covariance bounds; and
Remark~\ref{rem:onepoint} shows one-point statistics cannot suffice.
\item[(iii)] \emph{Rank-zero rigidity.} Theorem~\ref{thm:rigidity} leaves open the case
$B=0$, where by Remark~\ref{rem:rankzero} the question is equivalent to a critical
velocity bound. The Liouville theorems available for Type-I ancient solutions
\cite{AlbrittonBarker2019} require a backward $L^3$ control that \eqref{eq:typeI} alone
does not supply, so they do not close this case directly. Whether the two-point structure
of Section~\ref{sec:twopoint} can supply the missing information is, in our view, the
central question raised by this circle of ideas.
\end{enumerate}

\subsection*{Acknowledgements}

\end{document}